\documentclass[12pt]{amsart}
\usepackage{amsmath,amsthm,amscd,amsfonts,amssymb,enumerate}
\usepackage{graphicx}
\usepackage{color}
\usepackage[colorlinks]{hyperref}
\usepackage{amsfonts,amssymb,amscd,amsmath,enumerate,url,verbatim}
 \usepackage[dvips]{epsfig}
 \usepackage[none]{hyphenat}
\usepackage{amsmath,amssymb,amsfonts,enumerate,amsthm}
 \usepackage{amsgen, amstext,amsbsy,amsopn, amsthm, amsfonts,amssymb,amscd,amsmat
 h,euscript,enumerate,url,verbatim,calc,xypic}
 \usepackage{latexsym}
 \usepackage{graphics}
 \usepackage{color}
\headsep=1cm \numberwithin{equation}{section}
\numberwithin{equation}{section}

\numberwithin{equation}{section}
\input xy
\xyoption{all}

\newtheorem{thm}{Theorem}[section]
\newtheorem{cor}[thm]{Corollary}
\newtheorem{con}[thm]{Convention}
\newtheorem{lem}[thm]{Lemma}
\newtheorem{prop}[thm]{Proposition}
\newtheorem{defn}[thm]{Definition}

\newtheorem{rem}[thm]{Remark}

\newcommand{\Ann}{\operatorname{Ann}\,}

\newcommand{\Hom}{\operatorname{Hom}\,}
\newcommand{\Ext}{\operatorname{Ext}\,}

\newcommand{\Max}{\operatorname{Max}\,}

\newcommand{\Ass}{\operatorname{Ass}\,}
\newcommand{\Assh}{\operatorname{Assh}\,}

\newcommand{\grad}{\operatorname{grade}\,}

\renewcommand{\dim}{\operatorname{dim}\,}

\newcommand{\cd}{\operatorname{cd}\,}

\newcommand{\Min}{\operatorname{Min}\,}
\newcommand{\Sup}{\operatorname{Sup}\,}

\newcommand{\Z}{\mbox{Z}}

\newcommand{\fa}{\mathfrak{a}}
\newcommand{\fb}{\mathfrak{b}}

\newcommand{\fm}{\mathfrak{m}}
\newcommand{\fp}{\mathfrak{p}}
\newcommand{\fq}{\mathfrak{q}}

\newcommand{\fx}{\mathfrak{x}}
\newcommand{\fc}{\mathfrak{c}}
\newcommand{\LH}{Lichtenbaum-Hartshorne Theorem}

\begin{document}
\bibliographystyle{amsplain}


\title[ Cohomological dimension ]
 { Cohomological dimension with respect to the linked ideals}

\bibliographystyle{amsplain}

     \author[M. jahangiri]{Maryam jahangiri$^1$}
     \author[kh. sayyari]{khadijeh sayyari$^2$}

\address{$^{1, 2}$ Faculty of Mathematical Sciences and Computer,
Kharazmi University, Tehran, Iran.}

\email {jahangiri@khu.ac.ir, std-sayyari@khu.ac.ir}

\keywords{Linkage of ideals,  local cohomology modules}

 \subjclass[2010]{13C40, 13D45.}


\begin{abstract}

  Let $R$ be a commutative Noetherian ring.
 Using the new concept of linkage of ideals over a module, we show that if $\fa$ is an ideal of $R$ which is
 linked by the ideal $I$, then $\cd(\fa,R) \in \{ \grad \fa, \cd(\fa, H^{\grad \fa}_ {\fc} (R)) + \grad \fa\}, $
  where $\fc : = \bigcap_{\fp \in \Ass \frac{R}{I}- V(\fa)}\fp$. Also, it is shown that
for every ideal $\fb$ which
is geometrically linked with $\fa,$ $\cd(\fa, H^{\grad \fb}_ {\fb} (R))$  does not depend on $\fb$.
\end{abstract}

\maketitle

\bibliographystyle{amsplain}

\section{Introduction}

Let $R$ be a commutative
Noetherian ring, $\fa$   be an ideal  of $R$ and
$M$ be an $R$-module. For each $i\in \Z$, $H^i_{\fa} (M)$ denotes the $i$-th local cohomology module of $M$ with respect to $\fa$ (our terminology on local cohomology modules comes from \cite{BS}). Vanishing of these modules is an important problem in this topic and it attracts lots of interest, see for example \cite{L} and \cite{O}.
One of the most various invariants in local cohomology theory is
the cohomological dimension of   $M$ with respect to the
ideal $\fa$, i.e. $$\cd(\fa , M) : = \Sup \{ i \in \mathbb{N}_0|
H^i_{\fa} (M) \neq 0 \}.$$
In this paper, we consider the cohomological dimension of $M$ with respect to the "linked ideals" over it.

Following \cite{PS}, two proper ideals $\fa$ and $\fb$ in a Cohen-Macaulay
local ring $R$ is said to be linked if there is a regular sequence
$\underline{\fx}$ in their intersection such that $\fa =
(\underline{\fx}) :_R \fb$ and $\fb = (\underline{\fx}) :_R \fa$.
In a recent paper, \cite{JS}, the authors
introduced the concept of linkage of ideals over a module and
studied some of its basic properties. Let   $\fa$ and $\fb$ be two non-zero ideals of $R$ and
$M$ denotes a non-zero finitely generated $R$-module. Assume that
$\fa M\neq M\neq \fb M$  and let $I\subseteq \fa \cap \fb$ be an
ideal generating by an $M$-regular sequence. Then the ideals $\fa$
and $\fb$ are said to be linked by $I$ over $M$, denoted by
$\fa\sim_{(I;M)}\fb$, if $\fb M = IM:_M\fa$ and $\fa M = IM:_M\fb $.
  This concept is the classical concept of linkage of ideals in
\cite{PS}, where $M= R.$ Note that
these two concepts do not coincide \cite[2.6]{JS} although, in some cases they do
(e.g. Example \cite[2.4]{JS}).  We can also characterize linked ideals over $R,$ see \cite[2.7]{JS2}.

As an application of this generalization, one may characterize Cohen-Macaulay modules in terms of the type of linked ideals over it, see \cite[3.5]{JS2}.

  In this paper, we consider the above
generalization of linkage of ideals over a module and, among other
things, study the cohomological dimension of an $R$-module $M$ with
respect to the ideals which are linked over $M$. In particular, in
 Theorem \ref{t5} we show that if $\fa$ is an ideal of $R$ which is
 linked by $I$ over $M$, then $$\cd(\fa,M) \in \{ \grad_M \fa, \cd(\fa, H^{\grad_M \fa}_ {\fc} (M)) + \grad_M \fa\},
 $$
  where $\fc : = \bigcap_{\fp \in \Ass \frac{M}{IM}- V(\fa)}\fp$.

  And
in Corollary \ref{c2}  it is shown that
for every ideal $\fb$ which
is geometrically linked with $\fa$ over $M$, $\cd(\fa, H^{\grad_M \fb}_ {\fb} (M))$ is constant and does not depend on $\fb$.

Also, we show that if $\cd(\fb, R)< \dim(R)$ for any linked ideal  $\fb$ over $R$, then $\cd(\fa, R)< \dim(R)$
  for any ideal  $\fa$ (Corollary \ref{c1}).

Throughout the paper, $R$ denotes a commutative Noetherian ring with $1\neq 0$, $\fa$ and $\fb$ are two non-zero proper ideals
of $R$ and $M$ denotes a non-zero finitely generated $R$-module.

\section{ Cohomological dimension}


The cohomological dimension of an $R$-module $X$ with respect to $\fa$ is defined
by $$\cd(\fa , X) : = \Sup \{ i \in \mathbb{N}_0| H^i_{\fa} (X) \neq 0 \}.$$ It is a significant invariant in
local cohomology theory and attracts lots of interest, see for example \cite{L} and \cite{O}. In this section, we study
this invariant via "linkage". We begin by the definition of our main tool.

\begin{defn}\label{F1}
 \emph{Assume that $\fa M\neq M\neq \fb M$
 and let $I\subseteq \fa \cap \fb$ be an ideal generated by an $M$-regular sequence.
  Then we say that the ideals $\fa$ and $\fb$ are linked by $I$ over $M$,
  denoted $\fa\sim_{(I;M)}\fb$, if $\fb M = IM:_M\fa$ and $\fa M = IM:_M\fb $.
  The ideals $\fa$ and $\fb$ are said to be geometrically linked by $I$ over $M$ if $\fa M \cap \fb M =
  IM$. Also, we say that the ideal $\fa$ is linked over $M$ if there
  exist ideals $\fb$ and $I$ of $R$ such that $\fa\sim_{(I;M)}\fb$.
  $\fa$ is $M$-selflinked by $I$ if $\fa\sim_{(I;M)}\fa$.
  Note that in the case where $M = R$, this concept is the classical concept of linkage of ideals in \cite{PS}.}
\end{defn}
The following lemma, which will be used in the next proposition, finds some relations between local cohomology modules of $M$ with respect to ideals which are linked over $M.$

\begin{lem}\label{l07}
 Assume that $I$ is an ideal of $R$ such that $\fa\sim_{(I;M)}\fb$. Then
  \begin{itemize}
    \item [ (i) ] $\sqrt{I + \Ann M} = \sqrt{(\fa \cap \fb) + \Ann M}$. In particular, $H^i_{\fa \cap \fb} (M) \cong H^i_ I (M)$, for all $i$.
    \item [ (ii) ]Let $I=0$. Then, $\sqrt{0:_{\frac{R}{\Ann M}}\fa}= \sqrt{\frac{\Ann \fa M}{\Ann M}}= \sqrt{\frac{\fb + \Ann_R M}{\Ann_R M}}.$ Therefore, $H^ i _{\Ann_R M :_R \fa}(M)\cong H^ i _{\Ann \fa M}(M)\cong H^ i _{\fb}(M)$. In other words, if $M$ is faithful, then $H^ i _{\fb}(M)\cong H^ i _{0:_R \fa}(M)$.
  \end{itemize}
 \end{lem}

\begin{proof}
\begin{itemize}
  \item [ (i) ] Let $r \in\fa \fb$. By the assumption, $rM \subseteq IM $. Therefore, in view of \cite[2.1]{M}, there exist an integer $n$ and $b_1 ,..., b_n \in I$ such that $(r^n + r^ {n-1}b_1 + ...+ b_n)M= 0$. This implies that $r+\Ann M \in \sqrt{\frac{I + \Ann M}{\Ann M}} $, as desired. Now, the result follows using \cite[4.2.1]{BS}.
  \item [ (ii) ] Let $r \in \Ann \fa M$. Then,
  by the assumption,  $rM \subseteq \fb M $ and using similar argument  in part (i) one can see that  $r+ \Ann M \in \sqrt{\frac{\fb + \Ann M}{\Ann M}}.$ Also, via the fact
  that $\fa \fb \subseteq \Ann M,$ $\frac{\fb + \Ann M}{\Ann M} \subseteq 0:_{\frac{R}{\Ann M}}\fa.$ This proves
  the desired equalities. Now, the isomorphisms between local cohomology modules follows using \cite [4.2.1]{BS}.

\end{itemize}

\end{proof}
\begin{prop}\label{t8}
 Let $I$ be an ideal of $R$ such that $\fa\sim_{(I;M)}\fb$ and set $t: = \grad_M I$. Then $ \cd (\fa + \fb, M) \leq  \Max \{ \cd(\fa, M), \cd(\fb, M), t+1 \}.$ Moreover, if $ \cd (\fa + \fb, M) \geq t+1,$ e.g. $\fa$ and $\fb$ are geometrically linked over $M$, then the equality holds.
 \end{prop}

\begin{proof}
Note that $H^i_ I (M) = 0$ for all $i \neq t$, by \cite[1.3.9 and 3.3.1]{BS}. Now, the result follows from \ref{l07} and using the Mayer-Vietoris sequence

\begin{eqnarray*}... \longrightarrow H^i_ {\fa \cap \fb} (M) \longrightarrow H^{i+1}_ {\fa + \fb} (M)\longrightarrow H^{i+1}_ {\fa} (M) \oplus H^{i+1}_ {\fb} (M) \longrightarrow H^{i+1}_ {\fa\cap \fb}(M)\longrightarrow ... \textbf{.}\end{eqnarray*}

\end{proof}
The following corollary, which is immediate by the above
proposition, shows that, in spite of \cite[21.22]{e}, parts of an
$R$-regular sequence can not be linked over $R.$
\begin{cor}\label{c11}
Let $(R, \fm)$ be local and $x_1, ... , x_n \in \fm$ be an $R$-regular sequence, where $n \geq 4$. Then $(x_{i_1}, ... , x_{i_j})\nsim (x_{i_{j+1}}, ... , x_{i_{2j}})$, for all $1 < j \leq [\frac{n}{2}]$ and any permutation $(i_1,...,i_{2j}) $ of $\{ 1,..., 2j\}$ .
\end{cor}

Let $M \neq \fa M$. It is well-known,  by \cite[1.3.9]{BS}, that $\grad_M \fa \leq \cd(\fa, M).$ Then $M$ is said to be relative Cohen-Macaulay with respect to $\fa$ if $$\cd(\fa,M)= \grad_M \fa.$$

In the following proposition we compute the cohomological dimension of an $R$-module $M$ with respect to $\fa$ in two cases.
\begin{prop}\label{l5}
Let $I$ be an ideal of $R$ generating by an $M$-regular sequence of length $t$ and $\fa\sim_{(I;M)}\fb$.
\begin{itemize}
  \item [(i)] If $M$ is relative Cohen-Macaulay with respect to $\fa + \fb,$ then $H^i_{\fa} (M) =0$ for all $i \notin \{\grad_M \fa,\grad_M \fa+\fb \}.$
  \item [(ii)] If $I=0,$ then $\cd(\fa,M)= \cd(\fa, \frac{M}{\fb M}).$
\end{itemize}

 \end{prop}

\begin{proof}
\begin{itemize}
  \item [(i)] Using the assumption, $H^i_ {\fa \cap \fb} (M)= 0$ for all $i \neq t.$ Now, the result follows from the isomorphism $H^i_{\fa+\fb} (M) \cong H^i_{\fa} (M) \bigoplus H^i_{\fb} (M)$, for all $i>t+1$, and the surjective map $$H^{t+1}_ {\fa+\fb} (M)\longrightarrow H^{t+1}_ {\fa} (M) \oplus H^{t+1}_ {\fb} (M) \longrightarrow 0,$$ which are deduced by the Mayer-Vietoris sequence.
  \item [(ii)] It follows from the fact that $\fb M$ is $\fa$-torsion.
\end{itemize}

\end{proof}
The following lemma will be used in the rest of the paper.
\begin{lem}\label{l1}
Let $I$ be a proper ideal of $R$ such that $\fa\sim_{(I;M)}\fb$. Then, $\frac{M}{\fa M}$ can be embedded in finite copies of $\frac{M}{IM}$.
\end{lem}

\begin{proof}
Assume that $F \rightarrow \frac{R}{I} \rightarrow \frac{R}{\fb}
\rightarrow 0 $ is a free resolution of $\frac{R}{\fb}$ as
$\frac{R}{I}$-module. Then, using $^*:= \Hom_{\frac{R}{I}} (- ,
\frac{M}{IM})$, we get the exact sequence
 $0 \rightarrow (\frac{R}{\fb})^* \rightarrow (\frac{R}{I})^* \stackrel{f}{\rightarrow} F^*$,
 where $\frac{M}{\fa M}\cong Im (f) \subseteq F^* \cong \oplus \frac{M}{IM} $.

\end{proof}
The next theorem, which is our main result, provides a formula for $\cd(\fa,M)$ in the case where $\fa$ is linked over $M$.
\begin{thm}\label{t5}
Let $I$ be an ideal of $R$ generating by an $M$-regular sequence such that
 $\Ass \frac{M}{IM} = \Min Ass \frac{M}{IM}$ and $\fa$ is linked by $I$
 over $M.$ Then $$\cd(\fa,M) \in \{ \grad_M \fa, \cd(\fa, H^{\grad_M \fa}_ {\fc} (M)) + \grad_M \fa\}, $$
  where $\fc : = \bigcap_{\fp \in \Ass \frac{M}{IM}- V(\fa)}\fp$.
 \end{thm}

\begin{proof}
Note that, by \ref{l1}, $\Ass \frac{M}{\fa M}\subseteq \Ass \frac{M}{IM}$. Set $t:= \grad_M \fa$. Without loss of generality, we may assume that $\cd(\fa,M) \neq t$. Hence, there exists $\fp \in \Ass \frac{M}{IM}- V(\fa)$, else, $\sqrt{I + \Ann M}= \sqrt{\fa + \Ann M}$ which implies that $\cd(\fa,M)= t$. We claim that \begin{equation}\label{e2} \grad_M(\fa +\fc)> t.\end{equation} Suppose the contrary. So, there exist $\fp \in \Ass \frac{M}{IM}$ and $\fq \in \Ass \frac{R}{\fc}$ such that $\fa +\fq \subseteq \fp$. By the assumption, $\fp = \fq$ which is a contradiction to the structure of $\fc$.

          Let $A:= \{\fp|\fp \in\Ass \frac{M}{IM}\cap V(\fa)\}.$ Then, in view of \ref{l1},
          $$\sqrt{\fa+ \Ann M} = \bigcap_{\fp\in \Min Ass \frac{M}{\fa M}}\fp \supseteq \bigcap_{\fp\in A}\fp.$$
          On the other hand, let $\fp\in \Min A.$ Then, there exists $\fq\in \Min Ass \frac{M}{\fa M}$ such that $\fq\subseteq \fp.$ Hence,
          again by \ref{l1}, $\fq\in A$ and, by the structure of $\fp,$ $\fq=
          \fp.$ Therefore,
\begin{equation}\label{e1}
            \sqrt{\fa+ \Ann M} = \bigcap_{\fp\in A}\fp.
          \end{equation}
          Whence, using (\ref{e1}), it follows that $$\sqrt{I + \Ann M}= \bigcap_{\fp \in \Ass \frac{M}{IM}}\fp = \bigcap_{\fp \in \Ass \frac{M}{\fa M}}\fp \cap \fc =\sqrt{\fa \cap \fc+ \Ann M}.$$ Now, in view of (\ref{e2}), we have the following Mayer-Vietoris sequence \begin{equation}\label{e8} 0 \longrightarrow H^t_ {\fa} (M) \oplus H^t_ {\fc} (M) \longrightarrow H^t_ I (M)\longrightarrow N \longrightarrow 0\end{equation} for some $\fa$-torsion $R$-module $N$. Applying $\Gamma_{\fa}(-)$ on (\ref{e8}), we get the exact sequence $$  0 \rightarrow H^t_ {\fa} (M) \oplus \Gamma_{\fa}(H^t_ {\fc} (M)) \rightarrow \Gamma_{\fa}(H^t_ I (M))\rightarrow N \stackrel{f}\rightarrow H^1_{\fa}(H^t_ {\fc} (M)) \rightarrow H^1_{\fa}(H^t_ I (M))\rightarrow 0$$ and the isomorphism $$ H^i_{\fa}(H^t_ {\fc} (M)) \cong H^i_{\fa}(H^t_ I (M)), \textrm{for all }i>1.$$ Also, using \cite[3.4]{SA}, we have $H^{i+t}_{\fa}(M) \cong H^i_{\fa}(H^t_ I (M))$, for all $i\in \mathbb{N}_0$. This implies that

$$H^i_ {\fa} (M) \left\lbrace
           \begin{array}{c l}
              \cong H^{i-t}_{\fa}(H^t_ {\fc}(M)\ \ & \text{ \ \ $i>t+1$,}\\
              \cong \frac{H^1_{\fa}(H^t_ {\fc}(M)}{im(f)}\ \ & \text{ \ \ $i=t+1$,}\\
              \neq 0\ \   & \text{   \ \ ${i=t}$,}\\
              0\ \   & \text{   \ \ $\textrm{otherwise}$.}
           \end{array}
        \right.$$\\

Now, the result follows from the above isomorphisms.

\end{proof}

The following corollary, which follows from the above theorem,
provides a precise formula for $\cd(\fa,M)$ in the case where $\fa$
is geometrically linked over $M$ and shows how far $\cd(\fa,M)$ is from $\grad_M \fa.$ Note that by \cite[1.3.9]{BS}, $\grad_M \fa \leq \cd(\fa, M).$
\begin{cor}\label{c3}
Let $I$ be an ideal of $R$ generating by an $M$-regular sequence and $\fa$ and $\fb$ be geometrically linked by $I$ over $M$. Also, assume that $M$ is not relative Cohen-Macaulay with respect to $\fa.$ Then $$\cd(\fa,M) = \cd(\fa, H^{\grad_M \fa}_ {\fb} (M)) + \grad_M \fa.$$
\end{cor}
\begin{proof}
First, we show that \begin{equation}\label{e3}
  \{\fp|\fp \in \Ass \frac{M}{IM}- V(\fa)\}= \{\fp|\fp \in \Ass \frac{M}{IM}\cap V(\fb)\}.
\end{equation}
Let $\fp \in \Ass \frac{M}{IM}- V(\fa).$ Then, by \ref{l07}(i),
$\fp\in  V(\fb).$ On the other hand, if $ \fp\in \Ass
\frac{M}{IM}\cap V(\fb)$ then $\fp\nsupseteqq \fa,$ else, $$
0:_{\frac{M}{IM}}\fa\cap0:_{\frac{M}{IM}}\fb =
0:_{\frac{M}{IM}}\fa+\fb \neq 0,$$ which is a contradiction.

Now, in view of \cite[2.8(iii)]{JS},
$$\sqrt{\fb + \Ann M} =\bigcap_{\fp \in \Ass \frac{M}{IM}\cap V(\fb)}\fp.$$ This, in conjunction with (\ref{e3}),
 implies that $$\sqrt{\fb + \Ann M} =\bigcap_{\fp \in \Ass \frac{M}{IM}- V(\fa)}\fp.$$ Now, the result follows using similar argument as used in the proof of theorem \ref{t5}.
\end{proof}

\begin{prop}\label{t}
  Let $I$ be an ideal of $R$ generating by an $M$-regular sequence and $\fa$ and $\fb$ be geometrically linked by $I$ over $M$. Then $\grad_M \fa + \fb = \grad_M I + 1.$
\end{prop}

\begin{proof}
  Assume that $\fa$ and $\fb$ are geometrically linked by some $M$-regular sequence $I$ of length $t$ over $M.$ Then $\fa $ and $\fb$ are geometrically linked by zero over $\frac{M}{I M}.$ In view of the fact that $\grad _M (\fa +\fb) = \grad _{\frac{M}{IM}} (\fa +\fb)+ t,$ one can replace $M$ by $\frac{M}{I M}$ and assume that $I =0.$  Hence, it is enough to show that $\grad_M \fa + \fb = 1.$

   Let $\fp \in V(\fa + \fb).$ By \cite[2.12]{JS}, $\fa R_{\fp}$ and $\fb R_{\fp}$ are geometrically linked by zero over $M_{\fp}.$ In conjunction with the facts that $\grad_M \fa + \fb \leq \grad_{M_{\fp}} (\fa + \fb)R_{\fp}$ and $\grad_M \fa + \fb \geq 1,$ it is enough to show that $\grad_{M_{\fp}} (\fa + \fb)R_{\fp} = 1.$ Therefore, one may assume that $R$ is local.

   On the contrary, assume that $\grad_M \fa + \fb > 1.$ Using the long exact sequence $$0 \rightarrow \Hom (\frac{R}{\fa+\fb}, M)\rightarrow \Hom (\frac{R}{\fa}, M) \oplus \Hom (\frac{R}{\fb}, M) \rightarrow \Hom (\frac{R}{\fa\cap\fb}, M)\rightarrow \Ext_R^1(\frac{R}{\fa+\fb}, M)$$ and the assumption $(\fa \cap\fb)M\subseteq \fa M \cap\fb M = 0$, we get $$M \cong \fa M \oplus \fb M.$$ By the fact that $\fa \fb M=0$ we get, for any $i>0,$ $\fb^i M = \fb M.$ Then, by krull theorem, $\fb M =0$ and so $\fa M = M,$ which is a contradiction.
\end{proof}

\begin{rem}\label{e5}
\indent
\begin{itemize}
\item [ (i) ]
An ideal can be linked with more than one ideal. As an example, let $R$ be local and $x , y, z$ be an $M$-regular sequence. Then, $Rx$ is  geometrically linked with $Ry$ and $Rz$ over $M$.
\item[ (ii) ]
 Let $\fa$ and $\fb$ be geometrically linked by $I$ over $R$. Then, by \cite[2.4]{JS2}, $\sqrt{\fa}$ is linked by $I$. In particular, $\sqrt{\fa}$ is linked by every ideal $I'\subset \sqrt{\fa}$ which is generated by a maximal $R$-seqence in $\fa.$ Indeed, by \cite[2.8]{JS} and \cite[Theorem 1]{He},

 \begin{eqnarray*}
 \Ass \frac{R}{\fa} = V(\fa) \cap\Ass\frac{R}{I} = \Ass H^{\grad \fa}_ {\fa} (R)&=& \Ass\Hom_R(\frac{R}{\fa},\frac{R}{I'})\\ &=& V(\fa) \cap\Ass\frac{R}{I'},
 \end{eqnarray*}
which, in view of \cite[2.4]{JS2}, implies that $\sqrt{\fa}$ is linked by $I'.$
 \end{itemize}
\end{rem}
 The following corollary shows that for all ideals $\fb$ which are geometrically linked with $\fa$ over $M$, $\cd(\fa, H^{\grad_M \fb}_ {\fb} (M))$ is constant.��
\begin{cor}\label{c2}
  Let $\fa$ be  linked over $M$. Then, for every ideal $\fb$ which is geometrically linked with $\fa$ over $M$, $\cd(\fa, H^{\grad_M \fb}_ {\fb} (M))$ is constant. In particular,
  $$\cd(\fa, H^{\grad_M \fb}_ {\fb} (M))=\left\lbrace
           \begin{array}{c l}
               1,   &  ��\text{   M is relative Cohen-Macaulay}�\\�
                � &  �\text{with respect to} \ \fa,\\
               \cd(\fa, M)- \grad_M \fa,   & \text{   \textrm{otherwise.}}
           \end{array}
        \right.$$\\

 \end{cor}

\begin{proof}
Assume that $M$ is relative Cohen-Macaulay with respect to $\fa$ and $\fa$ and $\fb$ are geometrically linked by some $M$-regular sequence $I$ of length $t$ over $M.$ Then, by \cite[2.8]{JS}, we have the following Mayer-Vietoris sequence \begin{equation}\label{e.8} 0 \longrightarrow H^t_ {\fa} (M) \oplus H^t_ {\fb} (M) \longrightarrow H^t_ I (M)\longrightarrow N \longrightarrow 0\end{equation} for some $\fa$-torsion $R$-module $N$. Applying $\Gamma_{\fa}(-)$ on (\ref{e.8}), we get the exact sequence $$ H^{i-1}_ {\fa}( N) \rightarrow H^i_{\fa}(H^t_ {\fb} (M)) \rightarrow H^i_{\fa}(H^t_ I (M)),$$ for $i>1.$ Now, by \cite[3.4]{SA} and the assumption, we get $H^i_{\fa}(H^t_ {\fb} (M))= 0,$ for $i>1.$ On the other hand, again by \cite[2.8]{JS} , $\Gamma_{\fa}(H^t_ {\fb} (M))= 0.$

Therefore, using the convergence of spectral sequences $$H^i_{\fa}(H^j_ {\fb} (M)) \Rightarrow_i H^{i+j}_ {\fa+\fb} (M)$$ and the assumption, we get $H^1_{\fa}(H^t_ {\fb} (M))\cong H^{t+1}_ {\fa+\fb} (M).$ Now, by \ref{t}, $H^1_{\fa}(H^t_ {\fb} (M))\neq 0$ and $\cd(\fa, H^{\grad_M \fb}_ {\fb} (M))= 1.$

In the case where $M$ is not relative Cohen-Macaulay with respect to $\fa,$ the result follows from \ref{c3}.

\end{proof}
\begin{con}
 Assume that $I$ is an ideal of $R$ which is generated by an $M$-regular sequence. We define the set $$ S_{(I;M)} : = \{ \fa \triangleleft R | I \subsetneqq \fa,  \fa = IM:_R IM:_M \fa  \}.$$
$S_{(I;R)}$ actually contains of all linked ideals by $I$.
\end{con}

The following proposition, which is needed in the next two items, shows that any ideal $\fa$ with $\fa M \neq M$ can be embedded in a redical ideal $\fa'$ of $S_{(I;M)}$ for some $I.$
\begin{prop}\label{t7}
Assume that $\fa M \neq M$. Then,
\begin{itemize}
 \item [(i)] There exists an ideal $I$, generating by an $M$-regular sequence, such that $\fa$ can be embedded in a radical element $\fa'$ of $S_{(I;M)}$ with $\grad_M \fa' = \grad_M \fa= :t$. Also, $\fa'$ can be chosen to be the smallest radical ideal with this property.

 \item [(ii)] Let $\fa'$ be as in (i). Then $\Ass H^t_{\fa} (M) = \Ass \frac{R}{\fa'}$. In particular, $\fa'= \bigcap _{\fp \in \Ass H^t_{\fa} (M)} \fp$ and it is independent of the choice of the ideal $I.$
\end{itemize}
\end{prop}

\begin{proof}
\begin{itemize}
 \item [(i)] Let $x_1,...,x_t\in\fa$ be an $M$-regular sequence such that $\fa \subseteq Z_R (\frac{M}{(x_1,...,x_t)M})$. Replacing $(x_1,...,x_t^2)$ with $(x_1,...,x_t)$, we may assume, in addition, that $x_1,...,x_t\neq\fa$ and it is not a prime ideal. Set $\Lambda: = \{\fp | \fp \in \Ass (\frac{M}{(x_1,...,x_t)M}), \fa \subseteq \fp\}$ and $\fa' = \cap_{\fp \in \Lambda} \fp $. Then, setting $I:= (x_1,...,x_t)$ and using \cite[2.4]{JS2}, $\fa\subseteq\fa'$ is a radical ideal of $S_{(I;M)}.$

Assume that there exists a radical ideal $\fc\in S_{(I;M)}$ such that $\fa \subseteq \fc$. Hence, by \cite[2.4]{JS2}(iii), $\Ass\frac{R}{\fc} \subseteq \Ass\frac{M}{IM} \cap V(\fa) =\Lambda$. Therefore, $\fa' \subseteq \fc$.

\item [(ii)] By \cite[Theorem 1]{He} and \cite[1.4]{SW}, we have $$\Ass H^t_{\fa} (M) = \Ass \Ext^t_R(\frac{R}{\fa},M).$$ Also, using the above notations and in view of \cite[1.2.4 and 1.2.27]{BH},$$\Ass \Ext^t_R(\frac{R}{\fa},M)= \Ass\Hom_R(\frac{R}{\fa},\frac{M}{IM})= V(\fa) \cap\Ass\frac{M}{IM}= \Ass \frac{R}{\fa'_I}.$$ Now, by the above equalities, we have $\fa'= \bigcap _{\fp \in \Ass H^t_{\fa} (M)} \fp.$

\end{itemize}
\end{proof}

\begin{cor}\label{c4}
 Let $\fa$ be a linked ideal over $M$. Then, $\sqrt{\fa + \Ann M} = \bigcap _{\fp \in \Ass H^t_{\fa} (M)} \fp.$
\end{cor}
\begin{proof}
  It follows from \cite[2.4]{JS2}(v) and the above proposition.
\end{proof}
The following theorem provides some conditions in order to have $cd(\fa, M)< \dim M.$
\begin{thm}\label{t1}
 Let $(R,\fm)$ be local and $\underline{\fx} = x_1,...,x_t$ be an
 $M$-regular sequence of length $t$. Assume that $H^{\dim M}_{\fp} (M) =0$ for all $\fp \in \Ass_R \frac{M}{(\underline{\fx})M}$. Then, $H^{\dim M}_{\fa} (M) =0$ for any ideal $\fa \supseteq (\underline{\fx})$ with $\grad_M \fa = t.$
 \end{thm}

\begin{proof}

 Let $n := \dim M$ and $\fb \in S_{((\underline{\fx});M)}$ be a radical ideal. Then, by \cite[2.4]{JS2}, $\fb = \cap^l_{i=1} \fp_i$ for some $l\in\mathbb{N}$ and $\fp_1,...,\fp_l\in \Ass \frac{M}{(\underline{\fx})M}$. By the assumption, $H^n_ {\fb}(M)=0$ when $l= 1$. In the case $l>1$, set $\fc = \cap^l_{i=2} \fp_i$. Then, using the Mayer-Vietoris sequence

$$... \longrightarrow H^n_ {\fp_1} (M) \oplus H^n_ {\fc} (M) \longrightarrow H^n_ {\fb}(M)\longrightarrow 0$$ and the inductive hypothesis, we have $H^n_ {\fb} (M) = 0$.

Now, let $\fa \supseteq (\underline{\fx})$ be an ideal with $\grad_M \fa = t$. Then, by \ref{t7}, there exists a radical ideal $\fb\in S_{((\underline{\fx});M)}$ such that $\fa \subseteq \fb$. Let $\fb = \fa +(y_1,...,y_m).$ Then, using induction on $m$ and \cite[8.12]{BS}, it is straight forward to see that there exists an onto homomorphism $H^n_ {\fb} (M) \longrightarrow H^n_ {\fa}(M)\longrightarrow 0,$ and the result follows.

\end{proof}

\begin{rem}\label{r2}
Let the situations be as in the above theorem and assume, in
addition, that $(R,\fm)$ is complete. Let $\fa \supseteq
(\underline{\fx})$ be an ideal with $\grad_M \fa = t.$ Then, the \LH
\ shows that $\fa$ can not
 be coprimary with a member of $\Assh M$,i.e. there is no $\fp \in \Assh M$ with $\sqrt{\fa + \fp}= \fm$.
 \end{rem}

\begin{cor}\label{c1}
There is a linked ideal $\fb$ over
$R$ such that $H^{\dim R}_{\fb} (R) \neq 0.$ 
\end{cor}
\bibliographystyle{amsplain}

\end{document}